\documentclass{amsart}
\usepackage{amssymb, amsmath, latexsym}

\renewcommand{\baselinestretch}{\baselinestretch}
\renewcommand{\baselinestretch}{1.1}
\numberwithin{equation}{section}

\newcommand{\Z}{\mathbb Z}

\newcommand{\Q}{\mathbb Q}

\newcommand{\R}{\mathbb R}

\newcommand{\Hyper}{\mathbb H}
\newcommand{\A}{\mathbb A}

\newcommand{\br}{\mathbf r}
\newcommand{\gen}{\text{gen}}
\newcommand{\s}{\mathfrak s}
\newcommand{\n}{\mathfrak n}
\newcommand{\last}{\mathfrak l}

\newcommand{\ord}{\text{ord}}

\newtheorem{thm}{Theorem}[section]
\newtheorem{lem}[thm]{Lemma}
\newtheorem{cor}[thm]{Corollary}
\newtheorem{prop}[thm]{Proposition}

\theoremstyle{definition}

\theoremstyle{remark}
\newtheorem{rmk}[thm]{Remark}

\numberwithin{equation}{section}

\begin{document}

\title[Positive definite strictly $n$-regular quadratic forms]{A Finiteness theorem for positive definite strictly $n$-regular quadratic forms}

\author{Wai Kiu Chan}
\address{Department of Mathematics and Computer Science, Wesleyan University, Middletown CT, 06459, USA}
\email{wkchan@wesleyan.edu}

\author{Alicia Marino}
\address{Department of Mathematics and Computer Science, Wesleyan University, Middletown CT, 06459, USA}
\email{amarino@wesleyan.edu}




\begin{abstract}
An integral quadratic form is called strictly $n$-regular if it primitively represents all quadratic forms in $n$ variables that are primitively represented by its genus.  For any $n \geq 2$, it will be shown that there are only finitely many similarity classes of positive definite strictly $n$-regular integral quadratic forms in $n + 4$ variables.  This extends the recent finiteness results for strictly regular quaternary quadratic forms by Earnest-Kim-Meyer (2014).
\end{abstract}

\maketitle

\section{Introduction} \label{introduction}

The representation problem is arguably one of the most important problems in the arithmetic theory of quadratic forms.  Over a ring of arithmetic interest $R$, a quadratic form $f(\mathbf x)$ over $R$ in variables $\mathbf x = (x_1, \ldots, x_m)$ is said to represent another quadratic form $g(\mathbf y)$ over $R$ in variables $\mathbf y = (y_1, \ldots, y_n)$ if there exists a matrix $T \in M_{n\times m}(R)$ such that $g(\mathbf y) = f(\mathbf y T)$.  In other words, $g(\mathbf y)$ can be obtained from $f(\mathbf x)$ by a linear change of variables.  The representation is {\em primitive} if $T$ can be extended to a matrix in $\text{GL}_m(R)$.  The famous Hasse-Minkowski principle says that if an integral quadratic form $f$ represents another integral quadratic form $g$ over every completion of $\mathbb Q$, then $f$ must represent $g$ over $\Q$.  It is well-known that such a local-global principle does not hold in general if integral representations are considered instead.  The quadratic form $x^2 + 11y^2$, for example,  represents 3 (or the unary quadratic form $3x^2$) over every $p$-adic integers $\Z_p$ and over $\R$, but it clearly fails to represent 3 over $\Z$.   For an integral quadratic form $f$, the genus of $f$, denoted $\gen(f)$, is the set of quadratic forms that are equivalent to $f$ over every $\Z_p$ and over $\R$.  We say that a quadratic form $g$ is represented by $\gen(f)$ if $g$ is represented by some quadratic form in $\gen(f)$, which is tantamount to saying that $g$ is represented by $f$ over $\R$ and over every $\Z_p$, see \cite[102:5]{OM} or \cite[Theorem 6.1.3]{K}.  An integral quadratic form is called {\em $n$-regular} if it represents all the $n$-ary quadratic forms that are represented by its genus.  A quadratic form in $n$ or more variables and with only one equivalence class in its genus is  an example of an $n$-regular quadratic form.

Historically, 1-regular quadratic forms are simply called {\em regular}.  It was Dickson \cite{D} who first studied regular quadratic forms systematically, and later Watson \cite{WThesis, W54} showed that there are only finitely many similarity classes of positive definite regular ternary quadratic forms.  Enumeration of these similarity classes has been undertaken by several authors \cite{JKS, LO, Oh11}. The enumeration is almost completed, although the regularity of several candidates cannot be verified without assuming a form of the Generalized Riemann Hypothesis \cite{LO}.  Watson's finiteness results have been generalized to positive definite 2-regular quaternary quadratic forms by Earnest \cite{E94} and to positive definite $n$-regular quadratic forms in $n + 3$ variables ($n \geq 2$) by Chan-Oh \cite{CO03}.


Watson's finiteness result for regular ternary quadratic forms cannot be extended to quaternary quadratic forms, as Earnest exhibits in \cite{E95} an infinite family of similarity classes of positive definite regular quaternary quadratic forms.  However, he and his collaborators recently show in \cite{EKM14} that a finiteness result holds if primitive representations are considered.  A quadratic form is called {\em strictly $n$-regular} if it primitively represents all $n$-ary quadratic forms that are primitively represented by its genus.  As in the regular case, a strictly 1-regular quadratic form is simply called strictly regular.  In this terminology, the main result in \cite{EKM14} says that there are only finitely many similarity classes of positive definite strictly regular quaternary quadratic forms.  In this paper, we will obtain a similar finiteness result for strictly $n$-regular quadratic forms for {\em all} $n \geq 2$.  Our main result is:

\begin{thm} \label{mainthm}
For every $n \geq 2$, there are only finitely many similarity classes of positive definite strictly $n$-regular quadratic forms in $n + 4$ variables.
\end{thm}

Theorem \ref{mainthm} does not hold when $n = 1$.  For, in their classification of positive definite strictly regular diagonal quaternary quadratic forms \cite{EKM15}, Earnest-Kim-Meyer show that there are 27 similarity classes of positive definite strictly universal diagonal quaternary quadratic forms.  These quaternary quadratic forms primitively represent all positive integers.  As a result, there are infinitely many similarity classes of positive definite strictly regular quinary quadratic forms.

The subsequent discussion will be conducted in the geometric language of quadratic spaces and lattices, and any unexplained notations and terminologies can be found in \cite{K} and \cite{OM}.  Without causing much confusion, the quadratic map and its associated bilinear form on any quadratic space will be denoted by $Q$ and $B$, respectively.  For any $R$-lattice $M$, where $R$ is either $\Z$ or $\Z_p$, the scale of $M$ is the ideal $\s(M)$ generated by $B(x, y)$ for all $x, y \in M$, and its norm is the ideal $\n(M)$ generated by $Q(x)$ for all $x \in M$.  The $R$-lattice $M$ is called {\em integral} if $\s(M) \subseteq R$.

We say that an $R$-lattice $M$ is represented by another $R$-lattice $L$ if there exists an isometry $\sigma$ sending $M$ into $L$.  Such a $\sigma$ is called a representation of $M$ by $L$, and is called {\em primitive} if $\sigma(M)$ is a primitive sublattice of $L$, i.e. $\sigma(M)$ is a direct summand of $L$.

For any $R$-lattice $M$, we will write ``$M \cong A$" if $A$ is the Gram matrix of $M$ with respect to some basis of $M$.  A diagonal matrix with entries $a_1, \ldots, a_m$ on the main diagonal is denoted by $\langle a_1, \ldots, a_m \rangle$.  The group of units in $\Z_p$ is denoted by $\Z_p^\times$.  When $p$ is odd, the symbol $\Delta$ is always a nonsquare element in $\Z_p^\times$.  When we discuss lattices over $\Z_p$, $\Hyper$ denotes the hyperbolic plane and $\A$ stands for the anisotropic binary $\Z_p$-lattice, which is $\langle 1, -\Delta \rangle$ if $p > 2$, and $\left (\begin{smallmatrix} 2 & 1 \\ 1 & 2 \end{smallmatrix}\right )$ if $p = 2$.

From now on, the term lattice always refers to a positive definite integral $\Z$-lattice on a quadratic space over $\Q$.   The successive minima of a lattice $M$ are denoted by $\mu_1(M) \leq \cdots \leq \mu_m(M)$, where $m$ is the rank of $M$.  For each $k \leq m$, a $k\times k$ section of $M$ is defined to be the primitive sublattice of $M$ spanned by the first $k$ vectors in a Minkowski reduced basis of $M$.  The discriminant of $M$, denoted $d(M)$, is the determinant of any Gram matrix of $M$.  It is well-known (see for example \cite[103:4]{OM}) that there are only finitely many isometry classes of integral lattices whose discriminants are bounded above by an absolute constant.  The Hadamard inequality (see \cite{E94}), which says that $d(M) \leq \mu_1(M) \cdots \mu_m(M)$, allows us to bound the discriminant by establishing bounds on the successive minima in turn.

We say that a lattice $L$ is {\em $n$-regular} if it represents all $\Z$-lattices of rank $n$ that are represented by its genus; it is called {\em strictly $n$-regular} if it primitively represents all lattices that are primitively represented by its genus.  In this language, our main theorem is equivalent to saying that for every $n \geq 2$, there are only finitely many similarity classes of strictly $n$-regular lattices of rank $n + 4$.

One of the main tools in our argument is a set of regularity-preserving Watson transformations on lattices, which have been proven to be essential in all finiteness results on lattices that satisfy various kinds of regularity properties.  We will revisit these transformations in Section \ref{watson} and show that they preserve strict $n$-regularity, provided the lattices satisfy some specific local conditions (Proposition \ref{watson3}).   A lot of other useful properties of these transformations can be found in \cite[Section 2]{CO03}.  We will also use some of the results proved in \cite[Section 3]{CO03} for bounding the successive minima of a lattice which satisfies certain regularity conditions.    The strategy of the proof of Theorem \ref{mainthm} shares some common characteristics with the one used in proving the finiteness results for $n$-regular quadratic forms in \cite{CO03}.  However, the proof of Theorem \ref{mainthm} is largely self-contained and  this paper can be read as a companion of \cite{CO03}\footnote{Corrigenda to \cite{CO03} can be obtained from \texttt{http://wkchan.web.wesleyan.edu}.}.

\section{Preliminaries} \label{prelim}

In this section we will provide some preliminary lemmas which are tailored for the proof the main theorem.

\begin{lem} \label{strict}
Let $L$ be a strictly $n$-regular lattice.  Then
\begin{enumerate}
\item[(a)] $L$ is $n$-regular, and

\item[(b)] $L$ is strictly $(n-1)$-regular.
\end{enumerate}
\end{lem}

\begin{proof}
(a)  Let $M$ be a lattice of rank $n$ which is represented by the genus of $L$.  By the Hasse-Minkowski principle, we may assume that $M$ is on the space $\Q L$.  Let $S$ be the set of all primes $p$ at which $M_p$ is not a primitive sublattice of $L_p$.  For each $p \in S$, let $\sigma_p$ be an isometry of $\Q_p L$ such that $\sigma_p(M_p) \subseteq L_p$.  Since $S$ is a finite set, by \cite[81:14]{OM} there exists a lattice $N$ on $\Q M$ such that
$$N_p = \begin{cases}
M_p & \mbox{ if $p \not \in S$},\\
\sigma_p^{-1}(\Q_p(\sigma_p(M_p)) \cap L_p) & \mbox{ if $p \in S$}.
\end{cases}$$
Then $M$ is a sublattice of $N$, and $N_p = M_p$ is a primitive sublattice of $L_p$ for all $p \not \in S$.  Moreover, for every $p \in S$, $\sigma_p(N_p) = \Q_p (\sigma_p(M_p)) \cap L_p$ is a primitive sublattice of $L_p$.  Therefore, $N$ is primitively represented by $\gen(L)$, and it is represented by $L$ because of the strict $n$-regularity of $L$.  Since $M \subseteq N$, $M$ is represented by $L$ as well.

\noindent (b)  Let $K$ be a lattice of rank $n-1$ which is primitively represented by $\gen(L)$.  As in part (a), we may assume that $K$ is on $\Q L$.  Let $T$ be the set of primes $p$ at which $K_p$ is not a primitive sublattice of $L_p$.  This set $T$ is finite since $K_p \subseteq L_p$ and $K_p$ is unimodular at almost all primes $p$.  For every $p \in T$, let $\sigma_p$ be an isometry such that $\sigma_p(K_p)$ is a primitive sublattice of $L_p$.  Then there exists a lattice $N$ on $\Q L$ such that
$$N_p = \begin{cases}
L_p & \mbox{ if $p \not \in T$}\\
\sigma_p(L_p) & \mbox{ if $p \in T$}.
\end{cases}$$
It is clear that $K_p$ is a primitive sublattice of $N_p$ for all $p$, which implies that $K$ is a primitive sublattice of $N$.  Let $v$ be a vector in $N$ such that $K \oplus \Z[v]$ is primitive in $N$. Then $K \oplus \Z[v]$ is primitively represented by $\gen(L)$.  It follows from the strict $n$-regularity of $L$ that $K\oplus \Z[v]$ is primitively represented by $L$, which means that $K$ itself is primitively represented by $L$.
\end{proof}

Lemma \ref{strict} allows us to use results from \cite{CEO04, CO03} to obtain upper bounds on the first few successive minima of a strictly $n$-regular lattice.  We collect those results we need in the following two lemmas for the convenience of the readers.

\begin{lem}{\cite[Corollary 3.2]{CEO04}} \label{CEOlemma}
Let $L$ be a regular lattice of rank at least 4.  Then $\mu_3(L)$ is bounded above by an absolute constant.
\end{lem}

\begin{lem} \label{COlemma}
Let $L$ be a lattice.
\begin{enumerate}
\item[(a)] \cite[Lemma 3.3]{CO03}\footnote{In \cite[Lemma 3.3]{CO03}, $M$ can be any sublattice of $L$ of rank strictly less than the rank of $L$.  In its proof, the vector $y$ is a vector in $L$ such that $Q(y) \leq \mu_{k+1}(L)$ and $y \not \in \Q M$.} Let $M$ be a sublattice of $L$ of rank $k$.  If $k < \text{rank}(L)$ and $\n(M^\perp) \subseteq a\Z$ for some positive integer $a$, then $\mu_{k+1}(L) \geq \frac{a}{(dM)^2}$.

\item[(b)] \cite[Lemma 3.4]{CO03}  Let $M$ be a lattice of rank $k \geq 3$, and suppose that $L$ is $(k-1)$-regular of rank $> k$.  If $M$ is represented by $L$, then $\mu_{k+1}(L)$ is bounded above by a constant depending only on $M$.

\item[(c)] \cite[Lemma 3.5]{CO03} Let $M$ be a lattice of rank $k \geq 5$, and suppose that $L$ is $(k-2)$-regular of rank $> k$.  If $M$ is represented by $L$, then $\mu_{k+1}(L)$ is bounded above by a constant depending only on $M$.
\end{enumerate}
\end{lem}

The next lemma will be useful when the last two lemmas do not apply.

\begin{lem}\label{boundlemma}
Let $L$ be a lattice and $M$ be a primitive sublattice of $L$ of rank $m$.  Suppose that $K$ is a lattice of rank $k$ which is represented by $L$ but not by $M$. Then there exists a constant $C$, depending only on $k$ and $m$, such that
$$\mu_{m+1}(L) \leq C \max\{\mu_k(K), \mu_m(M) \}.$$
The same conclusion holds if $K$ is primitively represented by $L$ but not primitively by $M$.  Furthermore, the constant $C$ can be taken to be $1$ when $\max\{k, m \} \leq 4$.
\end{lem}
\begin{proof}
By replacing $K$ with an isometric copy in $L$, we may assume that $K$ is a sublattice of $L$ but not a sublattice of $M$.  Let $y_1, \ldots, y_k$ and $x_1, \ldots, x_m$ be Minkowski reduced bases of $K$ and $M$, respectively.  Then $y_j$ is not in $M$ for some $1\leq j \leq k$.  Since $M$ is a primitive sublattice of $L$, this $y_j$ must be outside of $\Q M$.  Consequently,
$$\mu_{m+1}(L) \leq \max\{Q(y_k), Q(x_m) \}.$$
By \cite[Chapter 12, Theorem 3.1]{Ca}, there exists a constant $C$, depending only on $k$ and $m$, such that
$$\mu_{m+1}(L) \leq C\, \max\{ \mu_k(K), \mu_m(M) \}.$$

As for the second assertion, we assume that $K$ is a primitive sublattice of $L$, but not a primitive sublattice of $M$.  This certainly means that $K$ cannot be a sublattice of $M$ at all.  We may then continue with the argument in the first paragraph.

The final assertion is a consequence of the fact that every lattice of rank $\ell \leq 4$ has a Minkowski basis $v_1, \ldots, v_\ell$ such that $Q(v_i)$ is the $i$-th successive minimum of that lattice; see \cite{van}.
\end{proof}

The following two lemmas are straightforward consequences of the theory of quadratic spaces and lattices.  Nonetheless, we provide their proofs  for the sake of completeness.

\begin{lem}\label{simple1}
Let $U$ be a positive definite quaternary space with a square discriminant.  Then there is a prime $p$ such that $U_p$ is anisotropic.
\end{lem}
\begin{proof}
Suppose that $U$ is isotropic at all finite primes $p$.  Then the Hasse invariant of $U_p$ is 1 if $p > 2$ and is $-1$ if $p = 2$; see \cite[Theorem 3.5.1]{K}.  Since the Hasse invariant of $U_\infty$ is 1, the product of all these local Hasse invariants is equal to $-1$. This contradicts the Hilbert Reciprocity.
\end{proof}

\begin{lem}\label{simple2}
Let $K$ be an anisotropic $\Z_p$-lattice.  There exists a positive integer $e = e(K)$ such that $K$ does not primitively represent any $p$-adic integer in $p^e\Z_p$.
\end{lem}
\begin{proof}
The set $P(K):= K\setminus pK$ is a compact subset of $K$.  By continuity, $Q(P(K))$ is also compact.  Suppose that for each positive integer $m$, there is an element $v_m$ in $P(K)$ such that $Q(v_m) \in p^m\Z_p$.  Then $\{Q(v_m)\}_{m \geq 0}$ is a Cauchy sequence whose limit is 0.  By compactness, 0 must be in $Q(P(K))$ which means that $K$ is isotropic.  This is a contradiction.
\end{proof}

The following is a simple but useful lemma in building primitive sublattices of a $\Z_p$-lattice.

\begin{lem} \label{localprim}
Let $L$ be a $\Z_p$-lattice and $N$ be a primitive sublattice of $L$.  Suppose that $v$ is a vector in $L$ such that $\Z_p[v]$ is an orthogonal summand of $L$ and $B(v, N) = 0$.  If $w \in L$ such that $B(w, \Z_p[v] \perp N) = 0$, then $N \perp \Z_p[v + w]$ is a primitive sublattice of $L$.
\end{lem}
\begin{proof}
Assume the contrary and suppose that for some $y \in L$, there exists a nonzero element $\alpha \in p\Z_p$ such that $\alpha y$ is a primitive vector in $N\perp \Z_p[v + w]$.  Then $\alpha y = x + a(v + w)$, where $x \in N$ and $a \in \Z_p$, and either $x$ is primitive in $N$ or $a \in \Z_p^\times$.  Since $\Z_p[v]$ splits $L$, $y$ can be expressed uniquely as $y = z + bv$, where $b \in \Z_p$ and $z$ is in the orthogonal complement of $v$ in $L$.  Thus,
$$b = \frac{a}{\alpha} \quad \mbox{ and } \quad z = \frac{x}{\alpha} + \frac{aw}{\alpha}.$$
The first equality implies that $a$ cannot be a unit.  Therefore, $x$ must be primitive in $N$, which means that $\frac{x}{\alpha}$ is not a vector in $L$ because $N$ is a primitive sublattice of $L$.  This is a contradiction, and hence $N \perp \Z_p[v + w]$ is primitive in $L$.
\end{proof}

\section{Watson Transformations} \label{watson}

For a lattice $L$ and a positive integer $m$, define
$$\Lambda_m(L) = \{x \in L : Q(x + z) \equiv Q(z) \mod m \mbox{ for all } z \in L \}.$$
If $p$ is a prime, then the $\Z_p$-lattice $\Lambda_m(L_p)$ is defined in an analogous manner.   It is clear that $\Lambda_m(L_p) = \Lambda_m(L)_p$.  These sublattices are introduced in \cite{CE04} as a generalization of the transformations used by Watson in \cite{WThesis, W62}.   We will make use of many of their properties that are given in \cite{CE04} and \cite{CO03}.  From now on, we follow the terminology introduced in \cite{CE04} to say that $L$ is {\em normalized} if $\n(L) = 2\Z$.\footnote{In \cite{CO03}, such a lattice is called {\em even primitive}.}  Every similarity class of lattices contains a unique isometry class of normalized lattices.

\begin{lem} \label{watson0}
If $L$ is a normalized lattice and $\s(L) = 2\Z$, then
$$\Lambda_4(L) = \{x \in L : Q(x) \in 4\Z\}.$$
\end{lem}
\begin{proof}
This is clear.
\end{proof}

\begin{lem}\label{watson1}
Let $L$ be a normalized lattice such that $Q(L_p) \neq 2\Z_p$ for a prime $p$.  Then
$$\Lambda_{2p}(L_p) = \{x \in L_p : Q(x) \in 2p\Z_p\}.$$
\end{lem}
\begin{proof}
We may suppose that $\s(L_2) = \Z_2$ when $p = 2$.  Then $L_p$ has an orthogonal splitting of the form $L_p = M_p \perp N_p$, where $M_p$ is anisotropic unimodular and $\s(N_p) \subseteq 2p\Z_p$.  By \cite[Lemma 2.1]{CO03}, $\Lambda_{2p}(L_p) = pM_p \perp N_p$, which is $\{x \in L_p : Q(x) \in 2p\Z_p\}$ in the present situation.
\end{proof}

\begin{lem} \label{watson2}
Let $p$ be a prime and let $L$ be a normalized lattice such that either $\s(L) = 2\Z$ when $p = 2$ or $Q(L_p) \neq 2\Z_p$.
\begin{enumerate}
\item[(a)] If $W$ is a subspace of $\Q_pL$, then $\Lambda_{2p}(W\cap L_p) = W \cap \Lambda_{2p}(L_p)$.

\item[(b)] If $\sigma: G \longrightarrow L$ is a representation of a lattice $G$ by $L$,  then $\sigma(\Lambda_{2p}(G)) = \Lambda_{2p}(\sigma(G))$.  Furthermore, if $\sigma$ is a primitive representation, then $\sigma(\Lambda_{2p}(G))$ is a primitive sublattice of $\Lambda_{2p}(L)$.
\end{enumerate}
\end{lem}
\begin{proof}
(a) Let $v$ be a vector in $\Lambda_{2p}(W \cap L_p)$.  Then $v \in W \cap L_p$ and $Q(v) \in 2p\Z_p$.  Therefore $v \in \Lambda_{2p}(L_p)$ by Lemmas \ref{watson0} and \ref{watson1}, and hence $v \in W \cap \Lambda_{2p}(L_p)$.  Conversely, let $w \in W \cap \Lambda_{2p}(L_p)$.  Then $Q(w + z) \equiv Q(z)$ mod $2p$ for all $z \in L_p$.  In particular, $Q(w + z) \equiv Q(z)$ mod $2p$ for all $z \in W\cap L_p$.  Thus, $w \in \Lambda_{2p}(W\cap L_p)$.

(b) For the first assertion, the inclusion $\sigma(\Lambda_{2p}(G)) \subseteq \Lambda_{2p}(\sigma(G))$ is clear because $Q(\sigma(x)) \equiv 0$ mod $2p$ for all $x \in \Lambda_{2p}(G)$.  The other inclusion is a simple consequence of the definitions of an isometry and the Watson transformations $\Lambda_{2p}$.

As for the second assertion, let $x$ be a vector in $\Lambda_{2p}(G)$.  Then $Q(\sigma(x)) = Q(x) \in 2p\Z$ and $\sigma(x) \in L$, which implies that $\sigma(x) \in \Lambda_{2p}(L)$ by Lemmas \ref{watson0} and \ref{watson1}.  Thus, $\sigma(\Lambda_{2p}(G))$ is a sublattice of $\Lambda_{2p}(L)$.

Let $y$ be a vector in $\Lambda_{2p}(L)$ such that $ay \in \sigma(\Lambda_{2p}(G))$ for some nonzero integer $a$.  Since $\sigma(\Lambda_{2p}(G)) \subseteq \sigma(G)$ and the latter is primitive in $L$, it must be the case that $y \in \sigma(G)$.  Additionally, since $y \in \Lambda_{2p}(L)$, the congruence $Q(y + z) \equiv Q(z)$ mod $2p$ holds for all $z \in L$, particularly for all those $z \in \sigma(G)$.  Therefore, $y$ is in $\Lambda_{2p}(\sigma(G)) = \sigma(\Lambda_{2p}(G))$.  This shows that $\sigma(\Lambda_{2p}(G))$ is in fact a primitive sublattice of $\Lambda_{2p}(L)$.
\end{proof}

\begin{prop} \label{watson3}
Let $p$ be a prime and let $L$ be a normalized strictly $n$-regular lattice such that either $\s(L) = 2\Z$ when $p = 2$ or $Q(L_p) \neq 2\Z_p$.  Then $\Lambda_{2p}(L)$ is also strictly $n$-regular.
\end{prop}
\begin{proof}
Let $N$ be a lattice of rank $n$ which is primitively represented by the genus of $\Lambda_{2p}(L)$.  At the prime $p$, there exists an isometry $\sigma_p$ of $\Q_pL$ and a primitive sublattice $H_p$ of $\Lambda_{2p}(L_p)$ such that $\sigma_p(H_p) = N_p$.  Let $G$ be the lattice on $\Q N$ defined by
$$G_q = \begin{cases}
N_q & \mbox{ if $q \neq p$},\\
\sigma_p(\Q_pH_p \cap L_p) & \mbox{ if $q = p$}.
\end{cases}$$
It is easy to see that $N \subseteq G$, and that $G$ is primitively represented by the genus of $L$.  By the strict $n$-regularity of $L$, we see that $G$ is primitively represented by $L$.  Let $\sigma: G \longrightarrow L$ be a primitive representation.  By Lemma \ref{watson2}(b), $\sigma(\Lambda_{2p}(G))$ is a primitive sublattice of $\Lambda_{2p}(L)$.  We claim that $N = \Lambda_{2p}(G)$, and this will complete the proof of the proposition.

At any prime $q \neq p$, $N_q = G_q = \Lambda_{2p}(G)_q$.  At $p$, we have
$$\Lambda_{2p}(G)_{p}  = \Lambda_{2p}(\sigma(\Q_p H_p \cap L_p)) = \sigma(\Lambda_{2p}(\Q_p H_p \cap L_p)) = \sigma(\Q_p H_p \cap \Lambda_{2p}(L_p)),$$
by Lemma \ref{watson2}.  Since $H_p$ is primitive in $\Lambda_{2p}(L_p)$, $\Q_p H_p \cap \Lambda_{2p}(L_p) = H_p$.  Thus, $N_p = \sigma(H_p) = \Lambda_{2p}(G)_p$.  As a result, $N = \Lambda_{2p}(G)$ as claimed.
\end{proof}

For any positive integer $m$ and positive rational number $r$, let $\Lambda_m^{(r)}(L)$ be the lattice obtained by scaling the quadratic map on $\Lambda_m(L)$ by $r$.   Given $m$, there is always a unique rational number $r$ such that $\Lambda_m^{(r)}(L)$ is normalized.   When $m = 2p$ and $L$ is normalized, then $2p^2\Z \subseteq \n(\Lambda_{2p}(L)) \subseteq 2p\Z$ and it follows that the rational number $r$ is either $\frac{1}{p}$ or $\frac{1}{p^2}$.   The Jordan decompositions of $L_p$ and $\Lambda_m^{(r)}(L)_p$ are certainly related, and the changes are recorded in \cite[Lemmas 2.1, 2.3, 2.4]{CO03}\footnote{There is a minor mistake in \cite[Lemma 2.3(2)]{CO03}.  When $\frac{dM}{4} \equiv 5$ mod 8, $\lambda_4(L)_2$ should be $M_2^3 \perp N_2^{\frac{1}{2}}$.  This affects neither the results in \cite{CO03} nor the conclusion we draw here. }.

Let $\br = \{r_1, \ldots, r_n\}$ be a finite sequence of rational numbers.  The composition of transformations $\Lambda_{2p}^{(r_n)}\circ \cdots \circ \Lambda_{2p}^{(r_1)}$ is abbreviated as $\Lambda_{2p}^{(\br)}$.  We say that $\br$ is admissible to $L$ and $p$ if for every $i$, $\Lambda_{2p}^{(r_i)}\circ \cdots \circ \Lambda_{2p}^{(r_1)}(L)$ is normalized.  Note that the entries of a sequence admissible to $L$ and $p$ are either $\frac{1}{p}$ or $\frac{1}{p^2}$.

The following corollary can be proved in the same way as \cite[Theorem 2.5]{CO03}, using Proposition \ref{watson3} together with \cite[Lemmas 2.1, 2.3, 2.4]{CO03} (see also \cite[Remark 2.7]{CO03}).

\begin{cor} \label{watson4}
Let $L$ be a normalized strictly $n$-regular lattice of rank $\geq 4$.  For any prime $p$, there exists a sequence $\br$ admissible to $L$ and $p$ such that the normalized lattice $\Lambda_{2p}^{(\br)}(L)$ is  strictly $n$-regular lattice with the properties that $Q(\Lambda_{2p}^{(\br)}(L)_p) = 2\Z_p$ and $\Lambda_{2p}^{(\br)}(L)_q$ is isometric to $L_q$ up to a scaling factor for all $q \neq p$.  If, in addition, $L_p$ is isotropic, we can make sure that $\Lambda_{2p}^{(\br)}(L)_p$ is split by $\Hyper$.
\end{cor}

\begin{rmk} \label{sublatticeminima}
Since $2pL \subseteq \Lambda_{2p}(L) \subseteq L$, we have
$$\mu_i(L) \leq \mu_i(\Lambda_{2p}(L)) \leq 4p^2 \mu_i(L)$$
for any $i$.  Thus, if $\ell$ is the length of a sequence $\mathbf r$ which is admissible to $L$ and $p$, then there are constants $\mathfrak c_1$ and $\mathfrak c_2$, depending only on $p$ and $\ell$, such that
$$\mathfrak c_1 \mu_i(L) \leq \mu_i(\Lambda_{2p}^{(\mathbf r)}(L)) \leq \mathfrak c_2 \mu_i(L).$$
\end{rmk}

\begin{rmk}\label{sublattice}
Suppose that $\br$ is sequence of length $\ell$ admissible to $L$ and a prime $p$.  Lemma \ref{watson2}(b) shows that if $M$ is a (resp. primitive) sublattice of $L$, then $\Lambda_{2p}^{(\br)}(M)$ is a (resp. primitive) sublattice of $\Lambda_{2p}^{(\br)}(L)$.  Furthermore,  there is a constant $\mathfrak c_3$, depending only on $p$ and $\ell$, such that $d(\Lambda_{2p}^{(\mathbf r)}(M)) \leq \mathfrak c_3 d(M)$.
\end{rmk}

\section{Bounding the first $n + 3$ minima}

We first prove a result bounding the first $n + 3$ successive minima of a normalized $n$-regular lattice of rank at least $n + 3$.  The case when $n = 3$ is fairly straightforward, and the proof for the $n = 2$ case is a modification of that for \cite[Theorem 1.1 ($n = 2$)]{CO03}.   We hope the detail provided here will help clarify some of the arguments used in the proof of \cite[Theorem 1.1 ($n = 2$)]{CO03} and make it more transparent to the readers.

\begin{prop} \label{mun+3}
Let $n \geq 2$ be an integer. If $L$ is a normalized $n$-regular lattice of rank $\geq n + 3$, then $\mu_{n+3}(L)$ is bounded above by an absolute constant.
\end{prop}
\begin{proof}
Let $L$ be a normalized $n$-regular lattice of rank $m \geq n + 3$.  In what follows, a numerical quantity is said to be bounded if its absolute value is bounded above by an absolute constant.

By Lemma \ref{strict}, $L$ is a regular lattice of rank $\geq 5$.  It follows from Lemma \ref{CEOlemma} that the first three successive minima of $L$ are bounded.   By Lemma \ref{COlemma}(b), $\mu_4(L)$ is also bounded.   If $n \geq 3$, then by Lemma \ref{COlemma}(c) $\mu_5(L), \ldots, \mu_{n+3}(L)$ are also bounded.

Henceforth, $L$ is always a normalized $2$-regular lattice of rank $m\geq 5$.  Let $M$ be a primitive quaternary sublattice of $L$ whose discriminant is bounded (for example, a $4\times 4$ section of $L$).  We will use $M$ to bound $\mu_5(L)$.  In below, we let $a$ be the positive generator of $\mathfrak n(M^\perp)$.  Here $M^\perp$ is the orthogonal complement of $M$ in $L$.

When $d(M)$ is a not a square, the argument in \cite[Page 2392]{CO03} carries over to here in verbatim.  Thus, we assume from now on that {\em $d(M)$ is a square}. For any prime $q$ and $\Z_q$-lattice $G$, we follow \cite{CO03} to use the notation $\last(G)$ to denote the norm ideal of the last component in a Jordan decomposition of $G$.

Suppose that there exist two primes $\ell \neq p$ such that $M_p$ and $M_\ell$ are anisotropic.  We claim that $\mu_5(L)$ is bounded under this circumstance. It is necessary that both $p$ and $\ell$ are divisors of $d(M)$ and hence they are bounded.  Fix a full rank sublattice $M' \cong \langle \alpha, \beta \rangle \perp B$ of $M$.  Write $\alpha = \ell^{\ord_\ell(\alpha)}\alpha'$ and $\beta = \ell^{\ord_\ell(\beta)}\beta'$ such that both $\alpha'$ and $\beta'$ are units in $\Z_\ell$.  Choose the smallest positive integer $b$ which satisfies:
\begin{enumerate}
\item[(1)] $b \equiv \ord_\ell(\alpha\beta)$ mod 2, and

\item[(2)] $b \geq \max\{ \ord_\ell(\last(M_\ell)), \ord_\ell(a) \}$.
\end{enumerate}
It is clear that $b - \ord_\ell(a)$ is bounded.  Next, let us choose a prime $s$ such that
\begin{enumerate}
\item[(3)] $-s \alpha'\beta' \in (\Z_\ell^\times)^2$, and

\item[(4)] if $q \neq \ell$ and $q \mid d(M')$, then $s\ell^b \in (\Z_q^\times)^2$.
\end{enumerate}
The choice of $s$ depends only on $M$ and hence $s$ can be chosen to be bounded.  Let $K$ be the binary lattice $\langle \alpha, \ell^{b + \delta}\beta s \rangle$, where $\delta = 4$ if $\ell = 2$ and $\delta = 0$ otherwise.  Then $\Q_\ell K \cong \Hyper$.  As a result, $\Q_\ell K$ is not represented by $\Q_\ell M$ and $K$ is not represented by $M$.

We claim that $K$ is represented by $L$.  It suffices to show that $K_q$ is represented by $L_q$ for all primes $q$ because $L$ is 2-regular.  By \cite[Theorems 1 and 3]{OM2}, $K_\ell$ is represented by $M_\ell \perp M_\ell^\perp \subseteq L_\ell$. Suppose that $q$ is not $\ell$.  If $q \nmid d(M')$, then $q\nmid d(M)$ and $M_q$ is unimodular.  In this case $M_q\cong \Hyper \perp \Hyper$, which certainly represents $K_q$.   Finally, if $q \mid d(M')$, then $s\ell^{b + \delta}$ is a square unit in $\Z_q$ and $K_q \cong \langle \alpha, \beta \rangle$ which is represented by $M_q$.

By Lemma \ref{COlemma}(a) and Lemma \ref{boundlemma}, we have
$$p^{\ord_p(a)} \ell^{\ord_\ell(a)} \leq a \leq d(M)^2\, \mu_4(M)\, \ell^{b + \delta}\alpha\beta s.$$
This shows that $p^{\ord_p(a)} \leq d(M)^2\, \mu_4(M)\, \ell^{b - \ord_\ell(a) + \delta}\alpha\beta s$.  Thus, $\ord_p(a)$ is bounded.  Now, choose an integer $\tilde{b}$ and a prime $\tilde{s}$ as before, satisfying conditions (1)--(4) with $p$ in place of $\ell$, and let $\tilde{K} = \langle \alpha, p^{\tilde{b} + \tilde{\delta}}\beta \tilde{s} \rangle$ where $\tilde{\delta} = 4$ if $p = 2$ and $\tilde{\delta} = 0$ otherwise.  Note that $\tilde{b}$ and $\tilde{s}$ are now bounded, and  $\tilde{K}$ is represented by $L$ but not by $M$.  Therefore, by Lemma \ref{boundlemma},
$$\mu_5(L) \leq \max\{\alpha, p^{\tilde{b} + \tilde{\delta}}\beta \tilde{s}, \mu_4(M) \},$$
and hence $\mu_5(L)$ is bounded as claimed.

By Lemma \ref{simple1}, there must be a prime $p$ at which $M_p$ is anisotropic, and now we may assume that $M_q$ is isotropic for all $q \neq p$.  We can apply the above argument and conclude that $q^{\ord_q(a)}$ is bounded.  Let
$$\mathcal B = \{ q \mbox{ prime } :   \mbox{ either } q \mid dM \mbox{ or } \ord_q(a) > 0\}.$$
Then all the primes in $\mathcal B$ are bounded and $p$ is in $\mathcal B$.  If $q \not \in \mathcal B$,  then $L_q$ is split by $\Hyper\perp \Hyper$ since in this case the rank of the unimodular component of a Jordan decomposition of $L_q$ is at least 5.

For the sake of convenience, we introduce the following notation.  Let $q$ be a prime  and $v_{1,q}, \ldots, v_{m,q}$ be a basis of a Jordan decomposition of $L_q$.  For $1 \leq i \leq m$, let $\gamma_{i,q}$ be the $q$-adic order of the norm ideal of the Jordan component that contains $v_{i,q}$.
In particular, $\gamma_{1, q} \leq \cdots \leq \gamma_{m, q}$.

Suppose that $q \neq p$ and $L_q$ is not split by $\Hyper$.  Then $q \in \mathcal B$, and $\gamma_{5, q}$ is also bounded because $q^{\ord_q(a)}$ is bounded.  By Corollary \ref{watson4}, there exists a sequence of rational numbers $\br$ of bounded length, admissible to $L$ and $q$,  such that $\Lambda_{2q}^{(\br)}(L)$ is split by $\Hyper$.  The sublattice $\Lambda_{2q}^{(\br)}(M)$ is primitive in $\Lambda_{2q}^{(\br)}(L)$ by Lemma \ref{watson2}(b).   The prime $p$ is the only prime for which $\Lambda_{2q}^{(\br)}(M)_p$ is anisotropic, and the discriminant of  $\Lambda_{2q}^{(\br)}(M)$ is also a square.   Let $a^{(\br)}$ be the positive generator of the norm ideal of the orthogonal complement of $\Lambda_{2q}^{(\br)}(M)$ in $\Lambda_{2q}^{(\br)}(L)$.  The set
$$\mathcal B^{(\br)}: = \{ \ell \mbox{ prime }: \ell\mid d \Lambda_{2q}^{(\br)}(M) \mbox{ or } \ord_\ell(a^{(\br)}) > 0\}$$
is either $\mathcal B$ or $\mathcal B \setminus \{q\}$.  By Remarks \ref{sublatticeminima} and \ref{sublattice}, we can replace $L$, $M$, $a$, and $\mathcal B$ by $\Lambda_{2q}^{(\br)}(L)$, $\Lambda_{2q}^{(\br)}(M)$, $a^{(\br)}$, and $\mathcal B^{(\br)}$, respectively, in our discussion.  We perform this procedure for every such prime $q$, after which we may assume that $L_q$ is split by $\Hyper$ for every $q \neq p$.

Suppose that $\gamma_{5,p} > \ord_p(\last(M_p)) + 4$.  Then, by \cite[Theorems 1 and 3]{OM2}, $M_p$ must be represented by $\Z_p[v_{1,p}, \ldots, v_{4,p}]$ which is the sum of a certain number of the Jordan components.  Then $\Z_p[v_{5,p}, \ldots,  v_{m,p}]$ is the the sum of the rest of the Jordan components of $L_p$.  By applying a suitable $\Lambda_{2p}^{(\br)}$ with an $\br$ of bounded length,  we may further assume that
$$L_p \cong \A\perp \A^p \perp \Z_p[v_{5,p}, \ldots,  v_{m,p}],$$
where $\A^p$ is the scaling of $\A$ by $p$.   Then $M^\perp_p$ is represented by $\Z_p[v_{5,p}, \ldots,  v_{m,p}]$, and hence $\ord_p(a) \geq \gamma_{5,p}$.  Let $\eta$ be the smallest positive even integer greater than or equal to $\ord_p(a)$, and let $\tilde{r}$ and $\tilde{t}$ be two primes satisfying the following conditions:
\begin{enumerate}
\item[(5)] $\tilde{r} \in -(\Z_q^\times)^2$ for all $q \in \mathcal B$,

\item[(6)] $\tilde{t} > 2p^{\eta - \ord_p(a)}\tilde{r}\,d(M)^2$ and $\tilde{t} \in (\Z_q^\times)^2$ for all $q \mid 2\tilde{r}$.
\end{enumerate}
It is clear  that $\tilde{r}$ and $\tilde{t}$ can be chosen to be bounded.  Let $\tilde{c} (< \tilde{t})$ and $\tilde{e}$ be integers such that $\tilde{t}\tilde{e} = \tilde{c}^2 + p^\eta\tilde{r}$, and consider the binary lattice
$$\tilde{N}: = \begin{pmatrix} 2\tilde{t} & 2\tilde{c} \\ 2\tilde{c} & 2\tilde{e} \end{pmatrix}.$$
It is easy to see that $\tilde{N}_q$ is represented by $L_q$ for all $q \neq p$.  By \cite[Theorems 1 and 3]{OM2}, $\tilde{N}_p$ is also represented by $L_p$.  It follows from the 2-regularity of $L$ that $\tilde{N}$ is represented by $L$.  However, $\tilde{N}$ is not represented by $M$ since $\tilde{N}_p$ is isotropic but $M_p$ is not.  Thus, by Lemma \ref{COlemma}(a) and Lemma \ref{boundlemma},
$$p^{\ord_p(a)} \leq a \leq d(M)^2 \, \max\{ 2\tilde{t}, 2\tilde{e}, \mu_4(M) \}.$$
If $2\tilde{e} \leq \max\{2\tilde{t}, \mu_4(M)\}$, then $\ord_p(a)$ is bounded.  Otherwise, $p^{\ord_p(a)} \leq 2\tilde{e}\,(dM)^2$.  Since $\tilde{t}\tilde{e} < \tilde{t}^2 + p^\eta \tilde{r}$, we have
$$p^{\ord_p(a)} < \frac{2\tilde{t}^2\,d(M)^2}{\tilde{t} - 2p^{\eta - \ord_p(a)}\tilde{r}\,d(M)^2},$$
and hence $\ord_p(a)$ is once again bounded.  In any event, we have shown that either $\gamma_{5,p} \leq \ord_p(\last(M_p)) + 4$ or else $\ord_p(a)$ is bounded and $\gamma_{5,p} \leq \ord_p(a)$.  Thus, $\gamma_{5,p}$ is bounded.

Now, by Corollary \ref{watson4}, we may apply the $\Lambda_{2p}^{(\mathbf r)}$ to $L$ with a sequence $\mathbf r$ of bounded length admissible to $L$ and $p$ and further assume that $L_p$ is split by $\Hyper$.  Let $r$ be a prime  such that $r \in -(\Z_q^\times)^2$ for all primes $q \in \mathcal B$.  This $r$ can be chosen to be bounded.  Choose another bounded prime $t$ such that $t \nmid d(M)$ and $t \in (\Z_q^\times)^2$ for all primes $q \mid 2r$.  Then $-r$ is a square modulo $t$, and there exist bounded integers $c (< t)$ and $e$ such that $te = c^2 + r$.  Let
$$N : = \begin{pmatrix} 2t & 2c \\ 2 c & 2e \end{pmatrix}.$$
Then $N_q$ is represented by $\Hyper$ for all $q \in \mathcal B$.  For $q \not \in \mathcal B$, $L_q$ is split by $\Hyper\perp \Hyper$ which represents any (even) binary $\Z_q$-lattice.  Thus, $N$ is represented by the genus of $L$, and hence by $L$ itself because of the 2-regularity of $L$.  However, $N_p$ is isotropic, meaning that $N_p$ is not represented by $M_p$.  Therefore, $N$ is not represented by $M$.  Consequently, by Lemma \ref{boundlemma}
$$\mu_5(L) \leq \max\{2t, 2e, \mu_4(M) \},$$
meaning that $\mu_5(L)$ is bounded.
\end{proof}

\section{Proof of Main Theorem ($n \geq 3$)}

In this section, we will present the proof of Theorem \ref{mainthm} when $n \geq 3$.  The case $n = 2$ will be discussed in the next section.

\begin{prop} \label{boundmin}
Let $L$ be a strictly $(n-3)$-regular lattice of rank $> n$.  If $L$ primitively represents a lattice $M$ of rank $n \geq 6$, then $\mu_{n+1}(L)$ is bounded above by a constant depending only on $M$.
\end{prop}
\begin{proof}
Let $p$ be a prime such that $p \nmid 2\, d(M)$.  When $n = 6$, we impose the additional condition that $-d(M)$ is a quadratic nonresidue modulo $p$.   Then
$$M_p \cong \Hyper \perp \Hyper \perp \A \perp T,$$
where $T$ is a unimodular $\Z_p$-lattice of rank $n - 6$ if $n > 6$, and 0 otherwise.  Thus, $M_p$ contains a primitive sublattice isometric to $\langle -p, p\Delta \rangle \perp T$.  Let $\{x_{1, p}, \ldots, x_{n-4, p}\}$ be an orthogonal basis of this sublattice.   Choose vectors $x_1, \ldots, x_{n-4}$ in $M$ such that $x_i$ approximates $x_{i, p}$ for $1 \leq i \leq n - 4$, and let $N: = \Z[x_1, \ldots, x_{n-4}]$.  The approximation is made fine enough that $N_p$ is a primitive sublattice of $M_p$ and $N_p \cong \Z_p[x_{1, p}, \ldots, x_{n-4, p}]$.   Note that $N$ depends only on $M$, and hence $d(N)$ is bounded above by a constant depending only on $M$.  Let $G$ be the orthogonal complement of $N$ in $M$.

Since $M_p$ is unimodular and $p > 2$, the orthogonal complements of $N_p$ and $\Z_p[x_{1,p}, \ldots, x_{n-4, p}]$ in $M_p$ are isometric \cite[Corollary 5.4.1]{K}.  Therefore, $G_p$ is isometric to $\langle 1, -\Delta, p, -p\Delta \rangle$, which is an anisotropic $\Z_p$-lattice representing every element in $\Z_p$.  However, $G_p$ does not primitively represent any element in $p^2\Z_p$.

We choose a positive integer $a$ which is represented by $G$ and such that $\ord_p(a) \geq 2$.  Thus for any nonnegative integer $e$, $p^ea$ is not primitively represented by $G_p$,  hence by \cite[Corollary 5.4.1]{K} again $N_p \perp \langle p^ea \rangle$ is not primitively represented by $M_p$.   However, we claim that either $N_p\perp \langle a \rangle$ or $N_p\perp \langle pa \rangle$ is primitively represented by $L_p$.

Let $v$ be a primitive vector of $L_p$ which is orthogonal to $M_p$.   Choose two vectors $w$ and $x$ in $G_p$ such that $Q(w) = -Q(v)$ and $Q(x) = a$.
Let
$$\delta = \begin{cases}
0 & \mbox{ if $\ord_p(a)\neq \ord_p(Q(v))$},\\
2 & \mbox{ if $\ord_p(a) = \ord_p(Q(v))$}.
\end{cases}$$
If $\ord_p(a) < \ord_p(Q(v))$, then $Q(x + v) = a + Q(v) = a \lambda^2$ for some $\lambda\in \Z_p^\times$ by the Local Square Theorem \cite[63:1]{OM}.   If, on the other hand, $\ord_p(a) \geq \ord_p(Q(v))$,
then
$$Q\left(\left(1 - p^\delta \frac{a}{Q(v)}\right)w + \left(1 + p^\delta \frac{a}{Q(v)}\right)v\right) = 4p^\delta a.$$
Thus, there is always a primitive vector $z$ of $L_p$ such that $Q(z) = p^\delta a$.  Moreover, this $z$ is of the form $z_0 + \eta v$, where $\eta \in \Z_p^\times$ and $z_0 \in M_p$, and $B(z_0, N_p\perp \Z_p[v]) = 0$.  Therefore, by Lemma \ref{localprim}, $N_p\perp \Z_p[z]$ is a primitive sublattice of $L_p$.

Note that, by our choice of $a$, there exists a vector $y \in G$ such that $Q(y) = p^\delta a$.  Let $K$ be the sublattice of $L$ defined by
$$K_q = \begin{cases}
N_p\perp \Z_p[y] & \mbox{ if $q = p$},\\
((\Q N \perp \Q[y]) \cap L)_q & \mbox{ if $q \neq p$}.
\end{cases}$$
Clearly, $K$ contains the full sublattice $N \perp \Z[y]$, and hence $d(K)$ is bounded above by a constant depending only on $M$.  As a result, the last successive minimum of $K$ is also bounded above by a constant depending only on $M$.    It is also clear from the construction that $K$ is primitively represented by the genus of $L$.  Consequently, $K$ is primitively represented by $L$, by virtue of the strict $(n-3)$-regularity of $L$.  However, since $K_p \cong N_p\perp \langle p^\delta a \rangle$ is not primitively represented by $M_p$, $K$ cannot be primitively represented by $M$.  By Lemma \ref{boundlemma},  $\mu_{n+1}(L)$ is bounded above by a constant depending only on $M$.
\end{proof}

\begin{proof}[Proof of Theorem \ref{mainthm} $(n \geq 3)$]
Let $L$ be a normalized strictly $n$-regular lattice of rank $n + 4$, $n \geq 3$.  By Lemma \ref{strict}, $L$ is an $n$-regular lattice of rank $\geq n + 4$.  Thus, by Proposition \ref{mun+3}, $\mu_1(L), \ldots, \mu_{n+3}(L)$ are bounded above by an absolute constant.  Let $M$ be an $(n+3)\times (n+3)$ section of $L$.  Note that $d(M)$ is bounded above by an absolute constant, and the number of possible isometry classes of $M$ is independent of $L$.   We then apply Proposition \ref{boundmin} to this $M$ and conclude that $\mu_{n+4}(L)$ is bounded above by an absolute constant.  Since $d(L) \leq \mu_1(L) \cdots \mu_{n+4}(L)$ which is bounded above by an absolute constant, the number of possible isometry classes of $L$ is finite.
\end{proof}


\section{Proof of Theorem \ref{mainthm} ($n = 2$)}

In this section, $L$ is always a normalized strictly 2-regular lattice of rank 6.   By Proposition \ref{mun+3}, the first five successive minima of $L$ are bounded above by an absolute constant.  We now move on to bound $\mu_6(L)$, which will lead to an absolute bound on $d(L)$ and Theorem \ref{mainthm} will follow.

Let $W$ be a quinary quadratic space.  It must represent its own discriminant, and hence it must contain a quaternary space $U$ whose discriminant is a square. The isometry class of $U$ is uniquely determined by $W$.  By Lemma \ref{simple1}, there is a prime $\ell$ such that $U_\ell$ is anisotropic.  We call $\ell$ a {\em characteristic} prime for $W$.  It is easy to see that the set of characteristic primes for $W$ is finite and depends only on the similarity class of $W$.

\begin{lem} \label{lemmaformu6}
Let $W$ be a quinary quadratic space and $\ell$ be a characteristic prime for $W$.  Let $H$ be a lattice on $W$.  If $\alpha$ is primitively represented by $H$ and $\alpha\, d(H)$ is a square, then there exists an even positive integer $e$, depending only on $\alpha$ and $H$, such that for any positive $\beta \in \ell^e\Z$, $\langle \alpha, \beta \rangle$ cannot be primitively represented by $H_\ell$.
\end{lem}
\begin{proof}
Let $\mathcal S$ be the set of primitive vectors $y$ in $H_\ell$ such that $Q(y) = \alpha$. For each $y \in \mathcal S$, let $G^{(y)}_\ell$ be the orthogonal complement of $y$ in $H_\ell$.   By Lemma \ref{simple2}, for each $y \in \mathcal S$, there exists a positive even integer $e(y)$ such that $G^{(y)}_\ell$  does not primitively represent any element in $\ell^{e(y)}\Z_\ell$.  Note that $\ord_\ell(d(G^{(y)}_\ell)) \leq \ord_\ell(\alpha) + \ord_\ell(d(H))$, and hence the number of possible isometry classes of $G^{(y)}_\ell$ is finite.  We can then let $e$  be the largest of all the $e(y)$.
\end{proof}

\begin{proof}[Proof of Theorem \ref{mainthm} ($n = 2$)]
Suppose that $L$ is a normalized strictly 2-regular lattice of rank 6.   As is in the proof of Proposition \ref{mun+3}, we say that a numerical quantity is bounded if its absolute value is bounded above by an absolute constant.  Since $\mu_5(L)$ is bounded, there exists a primitive quinary sublattice $H$ of $L$ whose discriminant is bounded.  Let $\ell$ be a characteristic prime for the space $\Q H$.  This $\ell$ depends only on the similarity class of $\Q H$, hence it is bounded.

Let us first suppose that  $\ord_\ell(\last(L_\ell)) > \ord_\ell(\last(H_\ell)) + 4$.  Let $L_\ell = L_0 \perp L_1$, where $L_1$ is the last component of a Jordan decomposition of $L_\ell$.  Since $H_\ell$ is represented by $L_\ell$, it is necessary that $H_\ell$ is already represented by $L_0$; see \cite[Theorems 1 and 3]{OM2}.   Therefore, by applying a suitable $\Lambda_{2\ell}^{(\br)}$ transformation with a sequence $\mathbf r$ of bounded length admissible to $L$ and $\ell$, we may assume that $L_\ell \cong \Hyper \perp N_\ell \perp \langle \ell^t \epsilon\rangle$, where $\ord_\ell(d(N_\ell))$ is bounded and $\langle \ell^t\epsilon \rangle$ is the last Jordan component of $L_\ell$ such that $t = \ord_\ell(\last(L_\ell))$.  We also replace $H$ by $\Lambda_{2\ell}^{(\mathbf r)}(H)$ whose underlying quadratic space is similar to $\Q H$.  As a result,  the prime $\ell$ is a characteristic prime for the space spanned by this new $H$.  By Remark \ref{sublattice}, $d(H)$ remains bounded after the replacement.

Fix a positive integer $\alpha$ which is primitively represented by $H$ and such that $\alpha\, d(H)$ is a square.  Choose a primitive vector $z$ of $H$ such that $Q(z) = \alpha$.  By Lemma \ref{lemmaformu6}, there exists a positive even integer $e$ such that for any positive integer $b$, $\langle \alpha, \ell^eb \rangle$ is not primitively represented by $H_\ell$.  Clearly, this $e$ is bounded because it depends only on $\alpha$ and $H$.  Let $w$ be a vector in the orthogonal complement of $z$ in $H$, and let $\beta = Q(w)$.   All of these together show that for any integer $f$, the $\Z_\ell$-lattice $\Z_\ell[z, \ell^{\frac{e}{2}} f w]$ is not primitively by $H_\ell$.

The lattice $L_\ell$ contains a primitive sublattice which is isometric to $\langle \alpha, -\alpha \rangle \perp N_\ell \perp \langle \ell^t\epsilon \rangle$.   Since the rank of $\langle -\alpha \rangle \perp N_\ell$ is 4, its ambient space must be universal and hence there is a positive integer $f$ such that $\ell^e\beta f^2$ is represented by $\langle -\alpha \rangle \perp N_\ell$.  This $f$ can be chosen to be bounded because the number of possible isometry classes of $N_\ell$ is bounded.   Now, let us further assume that $t >  e + \ord_\ell(\beta f^2) + 3$.   Then $\ell^e\beta f^2$ is primitively represented by $\langle -\alpha \rangle \perp N_\ell \perp \langle \ell^t\epsilon \rangle$ in a way that we can apply Lemma \ref{localprim} to claim that $\langle \alpha, \ell^e\beta f^2 \rangle$ is primitively represented by $L_\ell$.

Define a lattice $K$ on the space $\Q[z, w]$ by setting
$$K_q = \begin{cases}
\Z_\ell[z, \ell^{\frac{e}{2}} f w ] & \mbox{ if $q = \ell$},\\
(\Q[z, w]\cap L)_q & \mbox{ if $q \neq \ell$}.
\end{cases}$$
Note that $d(K)$ is bounded.  Moreover, $K$ is primitively represented by the genus of $L$, implying that $K$ is primitively represented by $L$ itself.  However, it follows from the construction that $K_\ell$ is not primitively represented by $H_\ell$, meaning that $K$ can never be primitively represented by $H$.  Therefore, by Lemma \ref{boundlemma}, there exists an absolute constant $C$ such that
$$\mu_6(L) \leq C\, \max\{\mu_2(K), \mu_5(H)\} \leq C\, \max\{\alpha, \ell^e \beta f^2, \mu_5(H) \},$$
which means that $\mu_6(L)$ is bounded.

So, we have shown that $\mu_6(L)$ is bounded as long as $\ord_\ell(\last(L_\ell))$ is greater than both $\ord_\ell(\last(H_\ell)) + 4$ and $e + \ord_\ell(\beta f^2) + 3$.  Now, let us assume that this is not the case; in particular, $\ord_\ell(\last(L_\ell))$ is bounded.  As is done before, we may assume that $L_\ell \cong \Hyper \perp J_\ell$ for some quaternary sublattice $J_\ell$.   We choose $\alpha, z$, and $w$ as before.  Since the dimension of $\Q_\ell J_\ell$ is 4, $\alpha c^2$ is primitively represented by $J_\ell$ for some $c \in \Q^\times$.  But now the number of possible isometry classes of $J_\ell$ is bounded.  Therefore, the number of such $c$ is also bounded.    By Lemma \ref{lemmaformu6}, we can find a bounded positive even integer $e'$ such that $\langle \alpha c^2, \ell^{e'} \beta \rangle$ is not primitively represented by $H_\ell$ for any $c$ in discussion.  Define a lattice $K'$ on $\Q[z, w]$ by setting
$$K'_q = \begin{cases}
\Z_\ell[cz, \ell^{\frac{e'}{2}} w] & \mbox{ if $q = \ell$},\\
(\Q[z, w]\cap L)_q & \mbox{ if $q \neq \ell$}.
\end{cases}$$
It is clear that $d(K')$ is bounded, $K_q'$ is primitively represented by $L_q$ for all $q \neq \ell$, and $K_\ell' \cong \langle \alpha c^2, \ell^{e'}\beta \rangle$ is not primitively represented by $H_\ell$.  However,  at $\ell$, note that $\ell^{e'}\beta$ is primitively represented by $\Hyper$.  Therefore, $K_\ell'$ is also primitively represented by $L_\ell$.  Thus, $K'$ is primitively represented by $L$ but $K'$ is not primitively represented by $H$.  By Lemma \ref{boundlemma},  $\mu_6(L)$ is bounded.
\end{proof}

\end{document}